\documentclass[11pt]{amsart}
\usepackage{amsmath,amsfonts,amssymb,amsthm, a4wide, url, mathscinet,mathdots}
\usepackage{hyperref}
\usepackage[usenames]{color}

\newtheorem{theorem}{Theorem}

\begin{document}
\title{A probabilistic variant of Sperner's theorem and of maximal $r$-cover free families}
\author{Noga Alon}
\address{Princeton University, NJ 08544, USA and Tel Aviv University, Tel Aviv 69978, Israel}
\author{Shoni Gilboa}
\address{The Open University of Israel, Raanana 43107, Israel.}
\author{Shay Gueron}
\address{University of Haifa, Haifa 31905, Israel, and Amazon Web Services, USA. }

\begin{abstract}
A family of sets is called $r$-\emph{cover free} if no set in the family is contained in the union of $r$ (or less) other sets in the family. A $1$-cover free family is simply an antichain with respect to set inclusion. Thus, Sperner's classical result determines the maximal cardinality of a $1$-cover free family of subsets of an $n$-element set. Estimating the maximal cardinality of an $r$-cover free family of subsets of an $n$-element set for $r>1$ was also studied.
In this note we are interested in the following probabilistic variant of this problem. Let $S_0,S_1,\ldots, S_r$ be independent and identically distributed random subsets of an $n$-element set. Which distribution minimizes the probability that $S_0\subseteq {\bigcup_{i=1}^r S_i}$? A natural candidate is the uniform distribution on an $r$-cover-free family of maximal cardinality. We show that for $r=1$ such distribution is indeed best possible. 
In a complete contrast, we also show that this is far from being true for every $r>1$ and $n$ large enough. 
\end{abstract}


\maketitle

\section{Introduction}\label{sec:intro}
For every positive integer $n$, let $\Omega_n$ be the set of all subsets of some fixed $n$-element set. 
For a positive integer $r$, a family $\mathcal{F}\subseteq \Omega_n$ is called $r$-\emph{cover free} if no set in $\mathcal{F}$ is contained in the union of $r$ (or less) other sets in $\mathcal{F}$. Let us denote by $g_r(n)$ the maximal cardinality of an $r$-cover free family in $\Omega_n$.
A $1$-cover free family in $\Omega_n$ is just an antichain in $\Omega_n$, with respect to set inclusion. Hence $g_1(n)=\binom{n}{\lfloor n/2\rfloor}$, by the classical result of Sperner (\cite{Sperner}).
For $r=2$  it was shown in \cite{EFF1} that 
$1.134^n<g_2(n)<O(\sqrt n) \left(\frac{5}{4}\right)^n$ and in the subsequent paper \cite{EFF2}, the same authors showed that for every $r$, 
\begin{equation}\label{eq:g_r+}\left(1+\frac{1}{4r^2}\right)^n<g_r(n)\leq\sum_{k=1}^n\frac{\binom{n}{\lceil k/r\rceil}}{\binom{k-1}{\lceil k/r\rceil-1}}.
\end{equation}
A different upper bound, which is better for large $r$, was obtained in \cite{DR}. In \cite{Ruszinko}, this bound was given a simpler proof and the following, more explicit, form: for every $r\geq 2$ and $n$ large enough,
\begin{equation}\label{eq:g_r large}g_r(n)\leq r^{8n/r^2}.
\end{equation}
We will now describe a probabilistic variant of $r$-cover free families of maximal cardinality. 
Let $\mathcal{P}_n:=\{p:\Omega_n\to[0,\infty):\sum_{A\in\Omega_n}p(A)=1\}$ be the family of probability distributions on $\Omega_n$.
For a positive integer $r$ and $p\in\mathcal{P}_n$, let $\tau_r(p)$ be the probability that $S_0\subseteq {\bigcup_{i=1}^r S_i}$, where $S_0,S_1,\ldots,S_r$ are random sets, drawn independently from $\Omega_n$ according to the distribution $p$.
Natural candidates to minimize $\tau_r$ are distributions in the set $\mathcal{CF}_{n,r}:=\{p\in\mathcal{P}_n: p \text{ is supported on an } r\text{-cover free family}\}$ 
(in which case, one only has to worry about choosing the same set twice). 

Clearly, $\min_{p\in\mathcal{CF}_{n,1}}\tau_1(p)=\frac{1}{\binom{n}{\lfloor n/2\rfloor}}$ where the minimum is attained for any distribution which is uniformly supported on a maximal antichain in $\Omega_n$.
Our first result is that for $n\geq 2$ this is indeed the minimum of $\tau_1$ over all $\mathcal{P}_n$.
\begin{theorem}\label{thm:r=1}
Suppose that $n\geq 2$. 
Then $\tau_1(p)\geq \frac{1}{\binom{n}{\lfloor n/2\rfloor}}$ for every $p\in\mathcal{P}_n$ and consequently, $\min_{p\in\mathcal{P}_n}\tau_1(p)=\min_{p\in\mathcal{CF}_{n,1}}\tau_1(p)$.
\end{theorem}
We note that the weaker statement that $\Pr(S_0\subseteq S_1 \text{ or } S_0\supseteq S_1)\geq\frac{1}{\binom{n}{\lfloor n/2\rfloor}}$ for every independent identically distributed random sets $S_0, S_1$ in $\Omega_n$, readily follows from the fact that $\Omega_n$ may be covered by $\binom{n}{\lfloor n/2\rfloor}$ chains (with respect to set inclusion). This symmetric version of Theorem \ref{thm:r=1} may be generalized as follows. 
For a property $P$ of families of sets, let $ex(n,P)$ denote the maximum possible cardinality of a family of sets in $\Omega_n$ satisfying $P$ and let $ex(n,k,P)$, for $0\leq k\leq n$, denote the maximum possible cardinality of a family of $k$-element sets in $\Omega_n$ satisfying $P$. 
Thus, for example, if $P_1$ is the property of being an antichain then $ex(n,P_1)={n \choose {\lfloor n/2 \rfloor}}$ by Sperner's Theorem,
if $P_2$ is the property of being an intersecting family and $n \geq 2k$ then
$ex(n,k,P_2)={{n-1} \choose {k-1}}$ by the Erd\H{o}s-Ko-Rado Theorem \cite{EKR},
and if $P_3$ is the property of not containing two sets whose symmetric difference has cardinality smaller than $d$, 
then $ex(n,P_3)$ is the maximum possible cardinality of an error correcting code with length $n$ and minimum distance $d$. 
Similarly, $ex(n,k,P)$ is the maximum cardinality of the corresponding constant weight code.

\begin{theorem}
\label{t11}
Let ${\mathcal H}$ be a family of unordered pairs of distinct sets in $\Omega_n$ and let $P_{\mathcal H}$ be the property of containing no pair from ${\mathcal H}$.
For $p\in\mathcal{P}_n$, let $\tau_{\mathcal H}(p):=\Pr(\{S_0,S_1\}\in{\mathcal H} \text{ or } S_0=S_1)$, where $S_0,S_1$ are random sets, drawn independently from $\Omega_n$ according to the distribution $p$.
Then $\min_{p\in \mathcal{P}_n}\tau_{\mathcal H}(p)=\frac{1}{ex(n,P_{\mathcal H})}$. 
Similarly, for every $0\leq k\leq n$, the minimum of $\tau_{\mathcal H}(p)$ over distributions $\mathcal{P}_n$ whose support is a subset of $\{A\in\Omega_n:|A|=k\}$ is $\frac{1}{ex(n,k, P_{\mathcal H})}$.
\end{theorem}
The examples mentioned above provide several specific applications of the theorem, and it is not difficult to describe others.

\medskip
In a complete contrast to Theorem \ref{thm:r=1}, we  show that for every $r>1$ (and $n$ large enough), the minimum of $\tau_r$ on $\mathcal{P}_n$ is much smaller than the minimum of $\tau_r$ over $\mathcal{CF}_{n,r}$. 
For every $0\leq\ell\leq n$, let $p_{\ell}$ be the probability distribution in $\mathcal{P}_n$ uniformly supported on the family of all $\ell$-element sets in $\Omega_n$.
\begin{theorem}\label{thm:r>1} 
Suppose that $r\geq 2$. There is $0<\mu_r<1$ such that for every $n$ large enough, 
$\min_{0<\ell<\frac{n}{r}}\tau_r(p_{\ell})<\mu_r^n\min_{p\in\mathcal{CF}_{n,r}}\tau_r(p)$
and consequently,
$\min_{p\in\mathcal{P}_n}\tau_r(p)<\mu_r^n\min_{p\in\mathcal{CF}_{n,r}}\tau_r(p)$.
\end{theorem}

For every $r\geq 2$, Theorem \ref{thm:r>1} shows that $\min_{p\in{\mathcal P}_n}\tau_r(p)$ is (much) smaller than $\min_{p\in\mathcal{CF}_{n,r}}\tau_r(p)$, which is at most 
$1-\left(1-\frac{1}{g_r(n)}\right)^r<\frac{r}{g_r(n)}$, as shown 
by considering any probability distribution uniformly supported on an $r$-cover free family of maximal cardinality. A lower bound for $\min_{p\in{\mathcal P}_n}\tau_r(p)$ is given in the following theorem.
\begin{theorem}\label{thm:lower}
Suppose that $r\geq 2$. There is $C_r>0$ such that $\min_{p\in\mathcal{P}_n}\tau_r(p)\geq\frac{C_r}{(g_r(n))^r}$ and hence, for $n$ large enough, by \eqref{eq:g_r large}, $\min_{p\in\mathcal{P}_n}\tau_r(p)\geq\frac{C_r}{r^{8n/r}}$.
\end{theorem}

We prove Theorems \ref{thm:r=1} and \ref{t11} in Section \ref{sec:r=1} 
and Theorems \ref{thm:r>1} and \ref{thm:lower} in Section \ref{sec:r>1}.

\section{The case $r=1$}\label{sec:r=1}
\begin{proof}[Proof of Theorem \ref{thm:r=1}]
Let $p\in\mathcal{P}_n$. 
Let $\mathcal{C}$ be the set of all maximal chains in $\Omega_n$, with respect to set inclusion. 
Every $A\in\Omega_n$ belongs to exactly $\frac{|\mathcal{C}|}{\binom{n}{|A|}}$ maximal chains. Therefore, 
$\frac{1}{|\mathcal{C}|}\sum_{C\in\mathcal{C}}\sum_{A\in C}\binom{n}{|A|}p(A)=\sum_{A\in \Omega_n}p(A)=1$
and since $\binom{n}{k}\leq \binom{n}{\lfloor n/2\rfloor}$ for every $0\leq k\leq n$,
\begin{equation}
\binom{n}{\lfloor \frac{n}{2}\rfloor}\sum_{A\in\Omega_n}p(A)^2\geq\sum_{A\in\Omega_n}\binom{n}{|A|}p(A)^2=\frac{1}{|\mathcal{C}|}\sum_{C\in\mathcal{C}}\sum_{A\in C}\binom{n}{|A|}^2 p(A)^2.\label{eq:A}
\end{equation}
Similarly, every pair $A_0\subsetneq A_1$ of sets in $\Omega_n$ belong to exactly $\frac{|\mathcal{C}|}{\binom{n}{|A_1|}\binom{|A_1|}{|A_0|}}$ maximal chains. Therefore, since $\frac{\binom{n}{k}}{\binom{\ell}{k}}\leq \frac{1}{2}\binom{n}{\lfloor \frac{n}{2}\rfloor}$ for every $0\leq k<\ell\leq n$,
\begin{align}
\binom{n}{\lfloor \frac{n}{2}\rfloor}\sum_{\substack{(A_0,A_1)\in\Omega_n^2\\ A_0\subsetneq A_1}}p(A_0)p(A_1)&\geq 2\sum_{\substack{(A_0,A_1)\in\Omega_n^2\\ A_0\subsetneq A_1}}\frac{\binom{n}{|A_0|}}{\binom{|A_1|}{|A_0|}}p(A_0)p(A_1)\nonumber\\
&=\frac{1}{|\mathcal{C}|}\sum_{C\in\mathcal{C}}\sum_{\substack{(A_0,A_1)\in C^2\\ A_0\neq A_1}}\binom{n}{|A_0|}\binom{n}{|A_1|}p(A_0)p(A_1).\label{eq:AA}
\end{align}
Summing up \eqref{eq:A} and \eqref{eq:AA} yields 
\begin{align*}
\binom{n}{\lfloor \frac{n}{2}\rfloor}\tau_1(p)&=\binom{n}{\lfloor \frac{n}{2}\rfloor}\sum_{\substack{(A_0,A_1)\in\Omega_n^2\\ A_0\subseteq A_1}}p(A_0)p(A_1)\geq\frac{1}{|\mathcal{C}|}\sum_{C\in\mathcal{C}}\sum_{(A_0,A_1)\in C^2}\binom{n}{|A_0|}\binom{n}{|A_1|}p(A_0)p(A_1)\\
&=\frac{1}{|\mathcal{C}|}\sum_{C\in\mathcal{C}}\left(\sum_{A\in C}\binom{n}{|A|}p(A)\right)^2\geq\left(\frac{1}{|\mathcal{C}|}\sum_{C\in\mathcal{C}}\sum_{A\in C}\binom{n}{|A|}p(A)\right)^2=1.
\end{align*}
as claimed.
\end{proof}

\begin{proof}[Proof of Theorem \ref{t11}] 
Let $G$ be the complement of the graph $(\Omega_n,\mathcal H)$. 
The size of the maximum clique in $G$ is clearly $ex(n,P_{\mathcal H})$. Therefore, by a theorem of Motzkin and Straus \cite[Theorem 1]{MS},
$$\min_{p\in\mathcal{P}_n}\tau_{\mathcal H}(p_{\mathcal H}) =1-2\max_{p\in\mathcal{P}_n}\sum_{\{A_0,A_1\} \text{ is an edge of }G}p(A_0)p(A_1)=\frac{1}{ex(n,P_{\mathcal H})}.$$
The second statement follows similarly, by considering the graph induced by $G$ on the vertex set $\{A\in\Omega_n:|A|=k\}$.
\end{proof}

\section{The case $r>1$}\label{sec:r>1}
Note that if $p\in\mathcal{P}_n$ is supported on an $r$-cover free family $\mathcal{F}$, then 
$$
1-\tau_r(p)=\sum_{F\in\mathcal{F}}p(F)\left(1-p(F)\right)^r\leq \sum_{F\in\mathcal{F}}p(F)\left(1-p(F)\right)\leq 1-\frac{1}{|\mathcal{F}|},
$$
and hence $\min_{p\in\mathcal{CF}_{n,r}}\tau_r(p)\geq\frac{1}{g_r(n)}$.
Therefore, to prove Theorem \ref{thm:r>1} for some $r\geq 2$, it is enough to show that there is $0<\mu_r<1$ such that for $n$ large enough,
\begin{equation}\label{eq:enough}
\min_{0<\ell<\frac{n}{r}}\tau_r(p_{\ell})<\mu_r^n\frac{1}{g_r(n)}.
\end{equation}

For large $r$ this may be easily deduced as follows.
For $\ell:=\lfloor \frac{n}{e r}\rfloor$, clearly
\begin{equation*}
\tau_r(p_{\ell})\leq\frac{\binom{r\ell}{\ell}}{\binom{n}{\ell}}\leq\left(\frac{r\ell}{n}\right)^{\ell}\leq \frac{1}{e^{\ell}}< e\frac{1}{e^{\frac{n}{e r}}}=e\left(e^{-\frac{1}{e}}r^{\frac{8}{r}}\right)^{\frac{n}{r}}\frac{1}{r^{\frac{8n}{r^2}}}.
\end{equation*}
Therefore, by \eqref{eq:g_r large}, for $n$ large enough
\begin{equation}\label{eq:r101}
\min_{0<\ell<\frac{n}{r}}\tau_r(p_{\ell})<e\left(e^{-\frac{1}{e}}r^{\frac{8}{r}}\right)^{\frac{n}{r}}\frac{1}{r^{\frac{8n}{r^2}}}<e\left(e^{-\frac{1}{e}}r^{\frac{8}{r}}\right)^{\frac{n}{r}}\frac{1}{g_r(n)}.
\end{equation}
It can be verified that $e^{-\frac{1}{e}}r^{\frac{8}{r}}<1$ for every $r\geq 101$. Thus, \eqref{eq:r101} confirms \eqref{eq:enough}, and hence Theorem \ref{thm:r>1}, for $r\geq 101$.
We proceed to describe the proof Theorem \ref{thm:r>1} for general $r\geq 2$.

\begin{proof}[Proof of Theorem \ref{thm:r>1}]
Let $\ell$ be an integer in the interval $[0,\frac{n}{r})$ for which $\binom{n}{\ell+1}/\binom{r\ell}{\ell}$ is maximal. 
It is simple to verify that if $n$ is large enough, then the sequence $\left(\binom{n}{j+1}/\binom{rj}{j}\right)_{j=0}^{\lfloor n/4r\rfloor+1}$ is increasing and hence $\ell>\frac{n}{4r}$.

Let $S_0,S_1,\ldots,S_r$ be random sets chosen, independently and uniformly, from all the $\ell$-element sets in $\Omega_n$.

Let $t:=\lfloor\ell^2/n\rfloor$ and let $\mathcal{E}$ be the event: $|{\bigcup_{i=1}^r S_i}|> r\ell-{t}$. 
It is easy to verify that the sequence $(\Pr(S_1\cup S_2=k))_{k=2\ell-t}^{2\ell}$ is decreasing, and hence
$$\Pr(\mathcal{E})\leq\Pr(|S_1\cup S_2|> 2\ell-t)\leq t\Pr(|S_1\cup S_2|=2\ell-t)=t\frac{\binom{n-\ell}{\ell-t}\binom{\ell}{t}}{\binom{n}{\ell}}.$$
Therefore, by \eqref{eq:g_r+},
\begin{align*}
\tau_r\left(p_{\ell}\right)&=\Pr\left(S_0\subseteq {\bigcup_{i=1}^r S_i}\right)\\
&=\Pr\left(\mathcal{E}\right)\Pr\left(S_0\subseteq {\bigcup_{i=1}^r S_i}~\bigg |~\mathcal{E}\right)+\Pr\left(\Omega_n\setminus \mathcal{E}\right)\Pr\left(S_0\subseteq {\bigcup_{i=1}^r S_i}~\bigg |~ \Omega_n\setminus \mathcal{E}\right) \\
&\leq t\frac{\binom{n-\ell}{\ell-t}\binom{\ell}{t}}{\binom{n}{\ell}}\cdot\frac{\binom{r\ell}{\ell}}{\binom{n}{\ell}}+1\cdot\frac{\binom{r\ell-t}{\ell}}{\binom{n}{\ell}}
=\left(t\frac{\binom{n-\ell}{\ell-t}\binom{\ell}{t}}{\binom{n}{\ell}}+\frac{\binom{r\ell-t}{\ell}}{\binom{r\ell}{\ell}}\right)\frac{n-\ell}{\ell+1}\cdot\frac{\binom{r\ell}{\ell}}{\binom{n}{\ell+1}}\\
&\leq \left(t\frac{\binom{n-\ell}{\ell-t}\binom{\ell}{t}}{\binom{n}{\ell}}+\frac{\binom{r\ell-t}{\ell}}{\binom{r\ell}{\ell}}\right)\frac{(n-\ell)n}{\ell+1}\cdot\frac{1}{g_r\left(n\right)},
\end{align*}
and \eqref{eq:enough} follows by using standard estimates on binomial coefficients. This completes the proof of the theorem.
\end{proof}

Finally, we prove Theorem \ref{thm:lower}.

\begin{proof}[Proof of Theorem \ref{thm:lower}]
Let $p\in\mathcal{P}_n$, let $N:=2g_r(n)$, let $S_1,\ldots,S_N$ be random sets, drawn independently from $\Omega_n$ according to the distribution $p$, and consider the random variable
$$I:=\{i\in [N]:\text{there is } J\subset[N]\setminus\{i\} \text{ of cardinality } r \text{ such that }S_i\subseteq\bigcup_{j\in J}S_j\}.$$
The family $\{S_i\}_{i\in [N]\setminus I}$ is clearly $r$-cover free, therefore $N-|I|=|[N]\setminus I|\leq g_r(n)$ and hence ${\mathbb E}|I|\geq N-g_r(n)=g_r(n)$.
On the other hand, clearly ${\mathbb E}|I|\leq N\binom{N-1}{r}\tau_r(p)$. Hence
\begin{equation*}
\tau_r(p)\geq\frac{g_r(n)}{N\binom{N-1}{r}}\geq\frac{r!\,g_r(n)}{N^{r+1}}=\frac{r!}{2^{r+1}g_r(n)^r}
\end{equation*}
and the result follows.
\end{proof}

\subsection*{Acknowledgements}
This research was partially supported by 
NSF grant DMS-1855464,
BSF grant 2018267,
the Simons Foundation,
a Google Research Award, 
NSF-BSF Grant 2018640,
the Israel Science Foundation (grant No. 1018/16), 
and 
the Center for Cyber Law and Policy at the University of Haifa, 
in conjunction with the Israel National Cyber Bureau in the Prime Minister's Office.

\end{document}